\documentclass{amsart}

\usepackage[mathscr]{euscript}
\usepackage{mathrsfs}
\usepackage{amsmath}
\usepackage{amsmath,amssymb,amsfonts}
\usepackage{bbm}
\def\1{\mathbbm{1}}
  \usepackage{paralist}
  \usepackage{graphics} %% add this and next lines if pictures should be in esp format
  \usepackage{epsfig}
  
   \newcommand{\R}{{\mathbb R}}
\newcommand{\C}{{\mathbb C}}                  % L'ensemble des complexes
\newcommand{\F}{{\mathbb F}}
\newcommand{\D}{{\mathbb D}}

\newcommand{\A}{{\mathbb A}}

\def\al{\alpha}
\def\om{\omega}
\def\Om{\Omega}
\def\ga{\gamma}
\def\si{\sigma}

\def\la{\lambda}

\def\t{\tau}
\def\La{\Lambda}

\def\calL{{\mathcal{L}}}

\def\calB{{\mathcal{B}}}

\def\calA{{\mathcal{A}}}

\def\calC{{\mathcal{C}}}

\def\calX{{\mathcal{X}}}

\def\calF{{\mathcal{F}}}

\def\R{\mathbb R}

\def\P{\mathbb P}
\def\D{\mathbb D}

\def\C{\mathbb C}

\def\E{\mathbb E}

\def\F{\mathbb F}

\def\L{\mathbb L}
\def\U{\mathbb U}

\newtheorem{theorem}{Theorem}[section]

\newtheorem{lemma}[theorem]{Lemma}
\newtheorem{proposition}{Proposition}

\theoremstyle{definition}
\newtheorem{definition}[theorem]{Definition}
\newtheorem{remark}{Remark}
\newtheorem{example}{Example}

\numberwithin{equation}{section}

\begin{document}

\title[Stochastic systems with input delay at the Boundary]{A semigroup approach to stochastic systems with input delay at the boundary}

\author{S. Hadd}
\address{Department of Mathematics,  Faculty of Sciences, Hay Dakhla, BP. 8106, 80000--Agadir, Morocco}
%\curraddr{Department of Mathematics,  Faculty of Sciences Agadir}
\email{s.hadd@uiz.ac.ma}

\author{F.Z. Lahbiri}
\address{Department of Mathematics, Faculty of Sciences, Hay Dakhla, BP. 8106, 80000--Agadir, Morocco}
\email{fatimazahra.lahbiri@gmail.com}

%    General info

%\date{January 1, 2001 and, in revised form, June 22, 2001.}



\begin{abstract}
This work focuses on the well-posedness of abstract stochastic linear systems with boundary input delay and  unbounded observation operators. We use product spaces and a semigroup approach to reformulate such delay systems into free-delay distributed stochastic systems with unbounded control and observation operators. This gives us the opportunity to use the concept of admissible control and observation operators as well as the concept of Yosida extensions to prove the existence and uniqueness of the solution process and provide an estimation of the observation process in relation to initial conditions and control process.  As an example, we consider a stochastic Schr\"odinger system  with input delay.
\end{abstract}

\maketitle

\section{Introduction}\label{sec:0}

The theory of infinite dimensional systems is now a well established field, see e.g. \cite{Cu-ZW}, \cite{Hadd-siam}, \cite{Sal}, \cite{Staf}, \cite{Tucsnak Weiss}, \cite{WeiRegu}. Recently, the author of \cite{Lu2015} extended some parts of this theory to stochastic linear systems of infinite dimension. In addition, the work \cite{LW-19} extends in a way the Salamon--Weiss systems to Stochastic Port-Hamiltonian Systems. More recently, in \cite{Lahbiri-Hadd} we introduced a theoretical approach to a large class of stochastic infinite dimensional systems. This approach uses the concept of Yosida extensions of admissible observation operators to give a representation to the output function of the system.  For other works dealing with stochastic infinite dimensional control systems we refer to \cite{Alb2} and \cite{Alb3}, \cite{Da-Za-SSR}, \cite{dap}, \cite{Fa-GO-09}, \cite{Go-Ma-17}, \cite{Lu-Zhang-21}, \cite{Lu-Zhang}.

This work is a natural continuation of the work \cite{Lahbiri-Hadd} which considers a stochastic linear system with input delays at the boundary conditions and a boundary observation of the form
\begin{align}\label{stoc-inputoutput}
\begin{cases}
dX(t)=A_{m}X(t)dt+M(X(t))dW(t),\quad X(0)=\xi, & t>0,  \\
GX(t)=\displaystyle\int_{-r}^{0}d\nu(\theta)U(t+\theta), & t>0,\\
U(t)=\varphi(t),\quad \P-a.s.& a.e.\; t\in [-r,0],\cr
Y(t)=\mathscr{C}X(t), & t>0.
\end{cases}
\end{align}
Here $A_m: D(A_m):=\mathscr{Z}\subset H\to H$ is a linear closed operator on a Hilbert space $H$, $M:H\to H$ is a linear bounded operator, $G:\mathscr{Z}\to \mathscr{U}$ and $\mathscr{C}:\mathscr{Z}\to \mathscr{Y}$ are linear (unbounded) operators, the trace and observation operator, respectively, with $\mathscr{U}$ and $\mathscr{Y}$ are Hilbert spaces,  $\nu:[-r,0]\to\calL(\mathscr{U})$ is a function of bounded variation continuous at zero with total variation $|\nu|$ (a positive Borel measure) such that $|\nu|([-\alpha,0])$ goes to zero as $\al\to 0$. On the other hand, we will work  with a standard one-dimensional  Brownian motion $(W(t))_{t\geq 0}$ together with a filtered probability space  ($\Om,\calF,\F,\P$) with a natural filtration $\mathbb{F}=\{\mathscr{F}_{t}\}_{t\geq0}$ generated by $(W(t))_{t\geq 0}$.  For $t<0$, $\mathscr{F}_{t}$ is taken to be $\mathscr{F}_{0}$. Finally, $X:[0,+\infty)\times \Om\to H,$ denotes the solution process of the system, $U:[-r,+\infty)\times \Om\to \mathscr{U},$  the control process, and $Y:[0,+\infty)\times \Om\to \mathscr{Y},$  the observation process of the system.

The main purpose of this work is to prove the well-posedness of the system \eqref{stoc-inputoutput}. We first notice that the case $\nu=B_1 \1_{\{0\}}$ with $B_1\in\mathcal{L}(\mathscr{U})$ ( so that the boundary condition takes the form $GX(t)=U(t)$), is recently investigated in \cite{Lahbiri-Hadd} and \cite{Lu2015}, \cite{Lu-Zhang-21},  \cite{Lu-Zhang}.

In the absence of the noise (i.e. in the deterministic case), input delay systems are well-studied using a natural feedback theory and/or a direct computation for some particular cases such as discrete delays, that is $\nu=\sum_{k=1}^n B_k \1_{[-r_k,0]}$ for $r_k\in [0,1]$ and $B_k\in\mathcal{L}(\mathscr{U})$, see e.g. \cite{BNS-21} \cite{CXC-19}, \cite{DLP-86}, \cite{Han}, \cite{XYL-06} and the references therein. We also mention that, another easy way to prove the well-posedness of deterministic case of \eqref{stoc-inputoutput} is the use of the work \cite[Section 5]{Hadd Manzo Ghandi}. In fact, with some effort one can use product spaces to reformulate the systems as a Cauchy problem with unbounded perturbation of the generator. In this case, the well-posedness is heavily related to the feedback theory of regular linear systems \cite{WeiRegu}.

In the presence of the noise the situation is quite different due to sensitivity of the stochastic convolution. In this case, the lack of a natural feedback theory for stochastic linear systems (see \cite{Lahbiri-Hadd}) makes the problem difficult. To overcome a such problem we will use a pure operator theory approach and boundary value problem (see e.g. \cite{Barbu-book-2013}, \cite{Greiner}) to convert the delay stochastic system \eqref{stoc-inputoutput} into a free-boundary input delay system (see the system \eqref{Trans-1}). This allows us to use the recent work \cite{Lahbiri-Hadd} to prove the well-posedness of the system \eqref{stoc-inputoutput} (see Theorem \ref{Main-1}).

In this work we will use the following notation. For any  Hilbert space $\calX$ and any real number $\t\in (0,\infty)$, we denote by $L^{2}_{\F}([0,\t],\calX)$ the set of all $\F$-adapted ($\calX$ valued)  processes $\zeta:[0,\t]\times\Omega\rightarrow\calX$ such that $\mathbb{E}\|\zeta\|^2_{L^{2}([0,\t],\calX)}<\infty.$ We denote by $L^{2}_{\calF}(\Omega,H)$ the Hilbert space of all $\calF$-measurable ($H$-valued) random variables $\xi:\Omega\rightarrow H$ satisfying $\mathbb{E}\|\xi\|_{H}^{2}<\infty$. With these notations, the initial conditions of the problem \eqref{stoc-inputoutput} are supposed as follow: $\xi\in L^2_{\mathscr{F}_0}(\Om,H)$ and $\varphi:[-r,0]\times \Om\to \mathscr{U}$  a $\mathscr{F}_0$-measurable process such that
\begin{align*}
\E\int^0_{-r}\|\varphi(s)\|^2ds<\infty.
\end{align*}

Some classical background on semigroup theory is also required. Let $A:D(A)\subset H\to H$ be the generator of a strongly continuous semigroup $(T(t))_{t\ge 0}$ on $H$. We denote by $\rho(A)$ the resolvent set of $A$, that is,  the set of all $\la\in\C$ such that the inverse $(\la I-A)^{-1}:=R(\la,A)$ exists (which is linear bounded due to the closed graph theorem). The spectrum of $A$ is the set $\si(A)=\C\backslash\rho(A)$. The type of the semigroup $(T(t))_{t\ge 0}$ is defined by $\om_0(A):=\inf_{t>0}\frac{1}{t}\log\|T(t)\|$. According to \cite[chap.II, Section 5]{Engel-Nagel},  if $\beta>\om_0(A)$, then $\{\la\in\C:{\rm Re}\la>\beta\}\subset \rho(A)$ and there exists $M\ge 1$ such that $\|T(t)\|\le Me^{\beta t}$ for any $t\ge 0$. On one hand, the domain  $D(A)$ endowed with the following graph norm $\|x\|_1:=\|x\|+\|Ax\|,\;x\in D(A),$ is a Banach space. On the other hand, for $\al\in\rho(A),$ we define a new norm $\|x\|_{-1}=\|R(\la,A)x\|$ for $x\in X$. This norm is independent of $\al$. We now define $X_{-1}$ as the completion of $X$ with respect to the norm $\|\cdot\|_{-1}$, which means that $X_{-1}=\overline{X}^{\|\cdot\|_{-1}}$. This space is called the extrapolation space associated with $A$ and $X$. As $H$ is a Hilbert space, then it is known that $X_{-1}$ is isomorphic to the topological dual $D(A^\ast)',$ where $A^\ast$ is the adjoint operator of $A$. Furthermore, the semigroup $(T(t))_{t\ge 0}$ is extended to another strongly continuous semigroup $(T_{-1}(t))_{t\ge 0}$ on $X_{-1},$ whose generator $A_{-1}:X\to X_{-1}$ is the extension of $A$ to $X$, see \cite[Chap.II, Theorem 5.5]{Engel-Nagel}.

The organization of this work is as follow: In Section \ref{sec:2}, we recall some facts about deterministic well-posed and regular linear systems. In Section \ref{sec:3}, we state and prove the main result of the work about the well-posedness of the system \eqref{stoc-inputoutput}. The last section is concerned with an example of the well-posedness of a stochastic Schr\"odinger
system  with input delay.

\section{A concise background on infinite dimensional linear systems}\label{sec:2}
The theory of infinite-dimensional linear systems with unbounded control and observation operators is now a well established theory which plays a key role in control engineering, see e.g. \cite{Cu-ZW}, \cite{Staf}, \cite{Tucsnak Weiss}. Recently, parts of this theory has been extended to stochastic systems of infinite dimension, see \cite{Lahbiri-Hadd}, \cite{LW-19}, \cite{Lu2015}, \cite{Lu-Zhang-21},\cite{Lu-Zhang}.

Throughout this section we use the following notations:  $(\mathfrak{X},|\cdot|)$ and $(\partial\mathfrak{X},\|\cdot\|_{\partial\mathfrak{X}})$ are Hilbert space, $\mathfrak{A}_m: D(\mathfrak{A}_m)\subset \mathfrak{X}\to \mathfrak{X}$ is a closed linear operator, and $\Theta: D(\mathfrak{A}_m)\to \partial\mathfrak{X}$ is a surjective linear operator.

\subsection{Admissible observation operators}\label{sub-c}
Consider the observed system
\begin{align}\label{AC}
\begin{cases}
\dot{z}(t)=\mathfrak{A}_m z(t),\quad z(0)=x,& t>0,\cr \Theta z(t)=0,& t\ge 0,\cr y(t)=\Upsilon z(t),& t\ge 0,
\end{cases}
\end{align}
where $\Upsilon:D(\mathfrak{A}_m)\to \partial \mathfrak{X}$ is a linear operator (not necessarily closed or closeable). In what follow, we assume that the linear operator
\begin{align}\label{A}
\mathfrak{A}:=\left(\mathfrak{A}_m\right)_{|D(\mathfrak{A})}\quad\text{with domain}\quad D(\mathfrak{A}):=\ker\Theta
\end{align}
generates a strongly continuous semigroup $(\mathfrak{T}(t))_{t\ge 0}$ on $\mathfrak{X}$. We consider the linear operator
\begin{align}\label{C}
\mathfrak{C}:=\Upsilon\quad\text{with}\quad D(\mathfrak{C}):= D(\mathfrak{A}).
\end{align}
The boundary system \eqref{AC} can be reformulated as the following distributed system
\begin{align}\label{AC-dist}
\begin{cases}
\dot{z}(t)=\mathfrak{A} z(t),\quad z(0)=x,& t>0,\cr y(t)=\mathfrak{C} z(t),& t\ge 0,
\end{cases}
\end{align}
Observe that for the initial condition $x\in\mathfrak{X}$, the function $z(t)=\mathfrak{T}(t)x,\;t\ge 0,$ is the mild solution of the differential equation in \eqref{AC-dist}. However, the expression $y(t;x)= \mathfrak{C}\mathfrak{T}(t)x$ is well defined only if $x\in D(\mathfrak{A})$. The observed linear system \eqref{AC-dist} (or \eqref{AC}) is called well-posed if for any $x\in \mathfrak{X},$ the output function $t\mapsto y(t;x)$ can be extended to a function (denoted by the same symbol) $y(\cdot;x)\in L^2_{loc}([0,+\infty),\partial\mathfrak{X})$ such that $\|y(\cdot;x)\|_{L^2([0,\al],\partial\mathfrak{X})}\le \ga |x|,$ for any $x\in \mathfrak{X}$ and $\al>0,$ where $\ga:=\ga(\al)>0$ is a constant.

In the following we give conditions for which the system \eqref{AC-dist} is well-posed. To this end, we need the following definition.
\begin{definition}
  The operator $\mathfrak{C}\in\calL(D(\mathfrak{A}),\partial\mathfrak{X})$ is called an admissible observation operator for $\mathfrak{A},$ if for some (hence all) $\al>0,$ there exists a constant $\ga:=\ga(\al)>0$ such that
  \begin{align}\label{Estimate-AC}
      \int^\al_0 \|\mathfrak{C} \mathfrak{T}(t)x\|^2_{\partial\mathfrak{X}}dt\le \ga^2 |x|^2
      \end{align}
      for any $x\in D(\mathfrak{A})$. In this case, we also say that $(\mathfrak{C},\mathfrak{A})$ is admissible.
\end{definition}
Now if $(\mathfrak{C},\mathfrak{A})$ is admissible then the system \eqref{AC-dist} is well-posed. In fact, the estimate \eqref{Estimate-AC} together with the density of $D(\mathfrak{A})$ in $\mathfrak{X}$ imply that the following map
\begin{align}\label{psi}
\Psi^{\mathfrak{A},\mathfrak{C}}: D(\mathfrak{A})\to L^2([0,\al],\partial\mathfrak{X}),\quad x\mapsto \Psi^{\mathfrak{A},\mathfrak{C}} x:=\mathfrak{C}\mathfrak{T}(\cdot)x=y(\cdot;x)
\end{align}
has a unique linear bounded extension to $\mathfrak{X}$ (denoted by the same symbol) $\Psi^{\mathfrak{A},\mathfrak{C}}\in\calL(\mathfrak{X},L^2([0,\al],\partial\mathfrak{X}))$. In order to give a representation of $\Psi^{\mathfrak{A},\mathfrak{C}},$ let us  consider the following operator (called the {\em Yosida extension} of $\mathfrak{C}$ for $\mathfrak{A}$),
\begin{align}\label{Yos-exten}
\begin{split}
D(\mathfrak{C}_\Lambda)&=\left\{ z\in \mathfrak{X}:\lim_{\la\to+\infty}\mathfrak{C} \la R(\la,\mathfrak{A})z\;\text{exists in}\;\partial\mathfrak{X}\right\}\cr \mathfrak{C}_\Lambda z&:= \lim_{\la\to+\infty}\mathfrak{C} \la R(\la,\mathfrak{A})z.
\end{split}
\end{align}
It is shown in \cite{Weiss3} that if $\mathfrak{C}$ is an admissible observation operator for $\mathfrak{A}$, then
\begin{align}\label{relation-Psi}
\mathfrak{T}(t)x\in D(\mathfrak{C}_\Lambda),\quad\text{and}\quad (\Psi^{\mathfrak{A},\mathfrak{C}} x)(t)=\mathfrak{C}_\Lambda \mathfrak{T}(t)x,\quad \forall x\in \mathfrak{X},\; a.e.\; t>0.
\end{align}
Thus if $(\mathfrak{C},\mathfrak{A})$ is well-posed, then the extended output function of the system \eqref{AC-dist} is
\begin{align*}
y(t;x)=\mathfrak{C}_\Lambda \mathfrak{T}(t)x
\end{align*}
for any $x\in X,$ and a.e. $t>0$.
\subsection{Admissible control operators}\label{sub-b}

Consider the following  boundary control system
\begin{align}\label{AB}
\begin{cases}
\dot{z}(t)=\mathfrak{A}_m z(t),\quad z(0)=x,& t>0,\cr \Theta z(t)=u(t),& t\ge 0,
\end{cases}
\end{align}
where $u:[0,+\infty)\to \partial\mathfrak{X}$ is the control function assumed to be a $2$-integrable function. According to \cite{Greiner}, the fact that $\Theta$ is surjective implies that the domain $D(\mathfrak{A}_m)$ can be viewed as the direct sum of $D(\mathfrak{A})$ and $\ker(\la-\mathfrak{A}_m)$ for any $\la\in\rho(\mathfrak{A})$. In particular, the following inverse exists \begin{equation*}
  \mathfrak{D}_{\lambda}:=(\Theta_{|\ker(\lambda-\mathfrak{A}_m)})^{-1}:\partial\mathfrak{X}\rightarrow \ker(\lambda-\mathfrak{A}_m),\qquad \la\in\rho(\mathfrak{A}).
\end{equation*}
The operator $\mathfrak{D}_\la$ is called the {\em Dirichlet operator} associated with $\mathfrak{A}_m$ and $\Theta$, and satisfies $\mathfrak{D}_{\lambda}\in\calL(\partial\mathfrak{X},\mathfrak{X})$ for any $\la\in \rho(\mathfrak{A})$.  We select the following operator
\begin{align}\label{B-hadd-rhandi}
\mathfrak{B}:=(\la-\mathfrak{A}_{-1})\mathfrak{D}_\la\in\calL(\partial\mathfrak{X},\mathfrak{X}_{-1}),
\end{align}
where $\mathfrak{A}_{-1}:\mathfrak{X}\to \mathfrak{X}$ is the generator of the extrapolation semigroup extension of $(\mathfrak{T}(t))_{t\ge 0}$ (see notations in the introductory section). As for any $v\in \partial\mathfrak{X}$ and $\la\in\rho(\mathfrak{A}),$ we have $\mathfrak{A}_m \mathfrak{D}_{\lambda} v=\la \mathfrak{D}_{\lambda} v,$ then $(\mathfrak{A}_m -\mathfrak{A}_{-1})\mathfrak{D}_{\lambda} v=(\la-\mathfrak{A}_{-1})\mathfrak{D}_{\lambda} v=\mathfrak{B}v$. As $\mathfrak{D}_\la$ is the inverse of $\Theta,$ we obtain
\begin{align}\label{sum}
\mathfrak{A}_m= (\mathfrak{A}_{-1}+\mathfrak{B}\Theta)_{|D(\mathfrak{A}_m)}.
\end{align}
Using the equality \eqref{sum}, the boundary system \eqref{ABC} can be reformulated as
\begin{align}\label{AB-dist}
\dot{z}(t)=\mathfrak{A} z(t)+\mathfrak{B} u(t),\quad z(0)=x,\; t>0.
\end{align}
The mild (or integral) solution of the nonhomogeneous equation  \eqref{AB-dist} is given by
\begin{align*}
z(t)=\mathfrak{T}(t)x+\int^t_0 \mathfrak{T}_{-1}(t-s)\mathfrak{B}u(s)ds
\end{align*}
for any $t\ge 0,\;x\in\mathfrak{X}$ and $u\in L^2([0,+\infty),\partial\mathfrak{X})$. This solution satisfies  $z(t)\in \mathfrak{X}_{-1}$ for $t>0$. Usually, we are interested in having $\mathfrak{X}$-valued solution of the equation \eqref{AB-dist}. To that purpose, we need the following definition.
\begin{definition}\label{Amisssible-control}
The operator $\mathfrak{B}\in\calL(\partial\mathfrak{X},\mathfrak{X}_{-1})$ is called an admissible control operator for $\mathfrak{A}$ if there exists $\tau>0$ such that
\begin{align*}
\Phi^{\mathfrak{A},\mathfrak{B}}_{\tau} u:=\int^{\tau}_0 \mathfrak{T}_{-1}(\tau-s)\mathfrak{B}u(s)ds\in\mathfrak{X}
\end{align*}
for any $u\in L^2([0,+\infty,\partial\mathfrak{X})$. In this case, we also say that $(\mathfrak{A},\mathfrak{B})$ is admissible.
\end{definition}
Using the closed graph theorem, it is shown in \cite{Weiss2} that if $(\mathfrak{A},\mathfrak{B})$ is admissible, then
\begin{align*}
\Phi^{\mathfrak{A},\mathfrak{B}}_{t}\in \calL (L^2([0,+\infty),\partial\mathfrak{X}), \mathfrak{X}),\qquad \forall t\ge 0.
\end{align*}
Thus if $\mathfrak{B}$ is an admissible control operator for $\mathfrak{A},$ the solution of \eqref{AB-dist} is continuous function from $[0,+\infty)$ to $\mathfrak{X}$.

\subsection{Well-posed linear systems}\label{sub-bc}
The object of this part it to define the concept of well-posedess for boundary control systems with a point observation. In fact, we consider the augmented system
\begin{align}\label{ABC}
\begin{cases}
\dot{z}(t)=\mathfrak{A}_m z(t),\quad z(0)=x,& t>0,\cr \Theta z(t)=u(t),& t\ge 0,\cr y(t)=\Upsilon z(t),& t\ge 0,
\end{cases}
\end{align}
where the operators  $\mathfrak{A}_m,\Theta$ and $\Upsilon$ are defined before as above. According to the previous two subsection, the system \eqref{ABC} can be reformulated as the following input-output distributed linear system
\begin{align}\label{ABC-dist}
\begin{cases}
\dot{z}(t)=\mathfrak{A} z(t)+\mathfrak{B} u(t),\quad z(0)=x,& t>0,\cr y(t)=\mathfrak{C}z(t),& t\ge 0,
\end{cases}
\end{align}
where the operators $\mathfrak{A},\mathfrak{C}$ and $\mathfrak{B}$ are given by \eqref{A}, \eqref{C} and \eqref{B-hadd-rhandi}, respectively.
\begin{definition}\label{well-posed-def}
We say that the system \eqref{ABC-dist} (or \eqref{ABC}) is well-posed if it admits a continuous mild solution $z:[0,+\infty)\to \mathfrak{X}$ and an output function $t\mapsto y(t;x,u)$ that can be extended to a locally $2$-integrable function on $[0,+\infty)$ such that
\begin{align*}
\|y(\cdot;x,u)\|_{L^2([0,\al],\partial\mathfrak{X})}\le \kappa \left(|x|+\|u\|_{L^2([0,\al],\partial\mathfrak{X})}\right)
\end{align*}
for any $x\in X,\al>0$ and $u\in L^2([0,\al],\partial\mathfrak{X})$, where $\kappa:=\kappa(\al)>0$ is a constant.
\end{definition}
\begin{remark}\label{R1}
 (i) As discussed in Subsection \ref{sub-b}, a sufficient condition for the existence of a continuous solution $z:[0,+\infty)\to \mathfrak{X}$ of the differential equation in \eqref{ABC-dist} is that $(\mathfrak{A},\mathfrak{B})$ is well-posed. \\
(ii) Even if $(\mathfrak{A},\mathfrak{B})$ is well-posed,  the expression $"\mathfrak{C}z(t)"$ is not well defined for a general initial condition $x\in \mathfrak{X}$ and all control $u\in L^2([0,+\infty),\partial\mathfrak{C})$. To overcome this obstacle, for any $t>0,$ we define the space
\begin{align*}
u\in W^{2,2}_{0,t}(\partial\mathfrak{X}):=\{u\in W^{2,2}([0,t],\partial\mathfrak{X}):u(0)=\dot{u}(0)=0\}.
\end{align*}
This space is dense in $L^2([0,t],\partial\mathfrak{X})$. Now assume (without loss of generality) $0\in\rho(\mathfrak{A})$ and take $u\in W^{2,2}_{0,t}(\partial\mathfrak{X})$. Then an integration by parts implies
\begin{align*}
\Phi^{\mathfrak{A},\mathfrak{B}}_t u&=\int^t_0 \mathfrak{T}_{-1}(t-s)(-\mathfrak{A}_{-1})\mathfrak{D}_0 u(s)ds\cr &= \left[ \mathfrak{T}(t-s)\mathfrak{D}_0 u(s)\right]^{s=t}_{s=0}-\int^t_0 \mathfrak{T}(t-s)\mathfrak{D}_0 \dot{u}(s)ds \cr &= \mathfrak{D}_0 u(t)-\int^t_0 \mathfrak{T}(t-s)\mathfrak{D}_0 \dot{u}(s)ds.
\end{align*}
Thus $\Phi^{\mathfrak{A},\mathfrak{B}}_t u\in D(\mathfrak{A}_m)$ for $t\ge  0$. Moreover, for any $(x,u)\in D(\mathfrak{A})\times W^{2,2}_{0,\al}(\partial\mathfrak{X}),$ for any $t\in [0,\al],$ we have
\begin{align*}
y(t)&=\Upsilon (\mathfrak{T}(t)x+\Phi^{\mathfrak{A},\mathfrak{B}}_t u)\cr & = \mathfrak{C}\mathfrak{T}(t)x+ \Upsilon \Phi^{\mathfrak{A},\mathfrak{B}}_t u \cr & =\left(\Psi^{\mathfrak{A},\mathfrak{C}} x+ {\bf F}^{\mathfrak{A},\mathfrak{B},\mathfrak{C}}u\right)(t),
\end{align*}
where $\Psi^{\mathfrak{A},\mathfrak{A}}$ is defined in \eqref{psi} and
\begin{align}\label{F}
\left({\bf F}^{\mathfrak{A},\mathfrak{B},\mathfrak{C}}u\right)(t)= \Upsilon \Phi^{\mathfrak{A},\mathfrak{B}}_t u.
\end{align}
\end{remark}
By combining Definition \ref{well-posed-def} and Remark \ref{R1} one can easily prove the following result.
\begin{proposition}\label{abc-well-posed}
 The input-output system \eqref{ABC-dist} (or \eqref{ABC}) is well-posed if and only if the following assertions hold
\begin{itemize}
  \item [{\rm (i)}] $(\mathfrak{A},\mathfrak{B})$ is admissible.
  \item [{\rm (ii)}] $(\mathfrak{C},\mathfrak{A})$ is admissible.
  \item [{\rm (iii)}] There exists a constant $\delta>0$ such that
\begin{align}\label{ineq-F}
\left\|{\bf F}^{\mathfrak{A},\mathfrak{B},\mathfrak{C}}u\right\|_{L^2([0,\al],\partial\mathfrak{X})}\le \kappa\|u\|_{L^2([0,\al],\partial\mathfrak{X})}
  \end{align}
  for any $u\in W^{2,2}_{0,\al}(\partial\mathfrak{X})$, where ${\bf F}^{\mathfrak{A},\mathfrak{B}}$ is the linear operator define in \eqref{F}. In this case we also say that $(\mathfrak{A},\mathfrak{B},\mathfrak{C})$ is well-posed triple.
\end{itemize}
\end{proposition}

If the system \eqref{ABC-dist} is well-posed, then the density of $W^{2,2}_{0,\al}(\partial\mathfrak{X})$ in $L^2([0,\al],\partial\mathfrak{X})$ and the inequality \eqref{ineq-F} imply that the operator ${\bf F}^{\mathfrak{A},\mathfrak{B},\mathfrak{C}}$ is extended to a linear bounded operator (denoted by the same symbol) ${\bf F}^{\mathfrak{A},\mathfrak{B},\mathfrak{C}}\in\calL(L^p([0,\al],\partial\mathfrak{X}))$ for any $\al>0$. Furthermore, the extended output function of the system \eqref{ABC-dist} is given by
\begin{align}\label{yyy}
\begin{split}y(t;x,u)&= \left(\Psi^{\mathfrak{A},\mathfrak{C}} x+\textbf{F}^{\mathfrak{A},\mathfrak{B},\mathfrak{C}}u\right)(t)\cr &=
\mathfrak{C}_\Lambda \mathfrak{T}(t)x+\left(\textbf{F}^{\mathfrak{A},\mathfrak{B},\mathfrak{C}}u\right)(t)\end{split}\end{align} for a.e. $t>0,$ and all $(x,u)\in \mathfrak{X}\times L^2_{loc}([0,+\infty),\partial\mathfrak{X})$.

%This property is very important if one wants to study feedback theory and stabilization of infinite dimensional linear systems.
%\begin{definition}\label{Def-well-posed-sys}
%The boundary system \eqref{ABC} (or equivalently the system \eqref{ABC-dist}) is called well-posed if the triple $(\mathfrak{A},\mathfrak{B},\mathfrak{C})$ is so.
%\end{definition}
%To give more properties of well-posed systems,
%Thus if the system \eqref{ABC-dist} is well-posed then its output function is given by
%\begin{align}\label{yyy}
%y(t;x,u)=\mathfrak{C}_\Lambda \mathfrak{T}(t)x+\left(\textbf{F}^{\mathfrak{A},\mathfrak{B},\mathfrak{C}}u\right)(t)
%\end{align}
%for a.e. $t>0,$ and all $(x,u)\in \mathfrak{X}\times L^2_{loc}([0,+\infty),\partial\mathfrak{X})$.

Next, in order to give a complete representation of the output function $y(\cdot)$ of the system \eqref{ABC-dist}, we need the following subclass of well-posed systems.
\begin{definition}\label{regular-triple}
 A well-posed triple $(\mathfrak{A},\mathfrak{B},\mathfrak{C})$ is called regular if
\begin{align*}
\lim_{t\to 0^+}\frac{1}{t} \int^t_0 \left({\bf F}^{\mathfrak{A},\mathfrak{B},\mathfrak{C}}\1_{[0,+\infty)}v_0\right)(s)ds=0
\end{align*}
for any $v_0\in \partial\mathfrak{X}$.
\end{definition}
This definition is also equivalent to  ${\rm Range}(\mathfrak{D}_\la)\subset D(\mathfrak{C}_\Lambda)$ for some (hence all) $\la\in\rho(\mathfrak{A})$.

According to \cite[Lem.3.6]{Hadd Manzo Ghandi}, if $(\mathfrak{A},\mathfrak{B},\mathfrak{C})$ is regular, then
\begin{align}\label{Hadd-hadd}
D(\mathfrak{A}_m)\subset D(\mathfrak{C}_\Lambda)\quad\text{and}\quad (\mathfrak{C}_\Lambda)_{|D(\mathfrak{A}_m)}=\Upsilon.
\end{align}
Furthermore, due to \cite{WeiRegu}, we also have
\begin{align*}
  \Phi^{\mathfrak{A},\mathfrak{B}}_tu\in D(\mathfrak{C}_\Lambda)\quad \text{and}\quad \left({\bf F}^{\mathfrak{A},\mathfrak{B},\mathfrak{C}}u\right)(t)=\mathfrak{C}_{\Lambda}\Phi^{\calA,\calB}_t u
  \end{align*}
  for a.e. $t\ge 0$ and all $u\in L^2([0,+\infty),\partial\mathfrak{X})$. Thus, in this case, the state trajectory of the system \eqref{ABC-dist} (equivalently of the system \eqref{ABC}) $z(t)=\mathfrak{T}(t)x+\Phi^{\mathfrak{A},\mathfrak{B}}_tu$ satisfies $z(t)\in D(\mathfrak{C}_\Lambda)$  and $y(t;x,u)=\mathfrak{C}_{\Lambda} z(t)$ for a.e. $t>0$ and all $(x,u)\in \mathfrak{X}\times L^2_{loc}([0,+\infty),\partial\mathfrak{X})$, due to \eqref{yyy}.

  In the following example we provide a regular system associated with the left shift semigroup on a Lebesque space. This example is important in the reformulation of the input delay stochastic system \eqref{stoc-inputoutput} to a free-delay stochastic system in product spaces (see Section \ref{sec:3}).

\begin{example}\label{shift-example}
Let $\mathscr{U}$ be a Hilbert space. On the Lebesgue space  $L^2([-r,0],\mathscr{U})$ (with $r>0$ is a fixed real number), we consider the operator $Q_m g=g'$ (the fist derivative) for $g\in W^{1,2}([-r,0],\mathscr{U})$. This operator is closed. We consider the dirac operator at zero $\delta_0:W^{1,2}([-r,0],\mathscr{U})\to \mathscr{U}$ such that $\delta_0 g=g(0),$ which is a surjective operator. Now consider the operator
\begin{align*}
Q g=g',\quad D(Q)=\ker\delta_0=\left\{g\in W^{1,2}([-r,0],\mathscr{U}):g(0)=0\right\}.
\end{align*}
It is well known (see e.g. \cite[Chap.II]{Engel-Nagel}) that $Q$ generates the left shift semigroup $S:=(S(t))_{t\ge 0}$ on $L^2([-r,0],\mathscr{U})$ defined by
\begin{align*}
(S(t)g)(\theta)=\begin{cases} 0,& -t\le \theta \le 0,\cr g(t+\theta),& -r\le \theta < -t.\end{cases}
\end{align*}
The Dirichlet operator associated with $Q_m$ and $\delta$ is $d_\la=e_\la$ for $\la\in\rho(Q)=\C,$ where $e_\la: \mathscr{U}\to L^2([-r,0],\mathscr{U})$ is such that $(e_\la v)(\theta)=e^{\la\theta}v$ for any $\theta\in [-r,0]$ and $v\in\mathscr{U}$. We select $\beta:=(\la-Q_{-1})d_\la\in\calL\left(\mathscr{U},(L^2([-r,0],\mathscr{U}))_{-1}\right)$. It is known that $\beta$ is an admissible control operator for $Q$ and the associated control maps are given by
\begin{align*}
\left(\Phi^{Q,\beta}_t u\right)(\theta)=\begin{cases} u(t+\theta),& -t\le \theta \le 0,\cr 0,& -r\le \theta < -t,\end{cases}
\end{align*}
for any $u\in L^2([0,+\infty),\mathscr{U})$, see \cite{HIR:06}. Now we define
\begin{align}\label{L}
L g=\int^0_{-r}d\mu(\theta)g(\theta),\qquad g\in W^{1,2}([-r,0],\mathscr{U}),
\end{align}
where $\mu:[-r,0]\to\calL(\mathscr{U})$ is a function of bounded variation with total variation $|\mu|$ satisfying $|\mu|([-\varepsilon,0])\to 0$ as $\varepsilon\to 0$. If we denote by $\L:=L_{|D(Q)}$ then $\L$ is an admissible observation operator for $Q,$ see \cite[Lem.6.2]{Hadd-SF}. Further more, according to \cite{HIR:06}, the triple $(Q,\beta,\L)$ is regular.
\end{example}

\section{Well-posedness of boundary input delay stochastic systems}\label{sec:3}
Throughout this section assume the following setting for the system \eqref{stoc-inputoutput}. The operators $A_m,G,\mathscr{M}$ and $\mathscr{C}$ are as in the introductory section. In addition, we suppose that $A:=(A_m)_{|D(A)}$ with $D(A)=\ker G$  generates  a strongly continuous semigroup $T:=(T(t))_{t\geq0}$ on the space $H$ and the boundary operator $G:\mathscr{Z}\to \mathscr{U}$ is surjective. Similarly to Section \ref{sec:2}, we consider the following Dirichlet operator associated with $A_m$ and $G,$ \begin{equation*}
  \mathbb{D}_{\lambda}:=(G_{|\ker(\lambda-A_{m})})^{-1}:\mathscr{U}\rightarrow \ker(\lambda-A_{m}),\qquad \la\in\rho(A).
\end{equation*}
Thus we have $\D_\la\in\calL(\mathscr{U},H)$ for any $\la\in \rho(A)$.  We select the following operators
\begin{align}\label{BC}
B:=(\la-A_{-1})\D_\la\in\calL(\mathscr{U},H_{-1})\quad\text{and}\quad C:=\mathscr{C}_{|D(A)}\in\calL(D(A),\mathscr{Y}).
\end{align}
With these notation we can write
\begin{align}\label{Am=A+BG}
A_m=\left( A_{-1}+BG\right)_{|\mathscr{Z}}.
\end{align}
For the control process $U:[-r,\infty)\times \Om\to \mathscr{U}$ and for each $t\ge 0,$ we denote by $U_t:[-r,0]\times \Om\to \mathscr{U}$ the history of $U$ defined by $U_t(\theta,\om)=U(t+\theta,\om)$ for a.e. $\theta\in [-r,0]$ and $\P$-a.s $\om\in \Om$.  Now we set $V(t,\theta,\om):=U_t(\theta,\om)$ for $t\ge 0,\;\theta\in [-r,0]$ and $\P$-a.s $\om\in\Om$. This process satisfies the following boundary control problem
\begin{align}\label{Shift-BC}
\begin{split}
& \frac{\partial}{\partial t}V(t,\cdot)=\frac{\partial}{\partial \theta}V(t,\cdot),\quad V(0,\cdot)=\varphi,\quad t>0,\cr
& V(t,0)=U(t),\quad t>0.
\end{split}
\end{align}
In the sequel we will transform the input delay system \eqref{stoc-inputoutput} to a free delay boundary control system. For this, we define the product space
\begin{align*}
\mathscr{H}=H\times L^2([-r,0],\mathscr{U})\quad \text{with norm}\quad \left\|\left(\begin{smallmatrix}x\\g\end{smallmatrix}\right)\right\|=\|x\|+\|g\|_2.
\end{align*}
We introduce the new state
\begin{align}\label{new-state}
Z(t)=\begin{pmatrix}X(t)\\U_t\end{pmatrix},\qquad t\ge 0.
\end{align}
We also define the following operators
\begin{align*}
& \mathscr{A}_m:=\begin{pmatrix}A_m&0\\0&Q_m\end{pmatrix},\quad D(\mathscr{A}_m):=\left\{\left(\begin{smallmatrix}x\\g\end{smallmatrix}\right)\in \mathscr{Z}\times W^{1,2}([-r,0],\mathscr{U}):Gx=Lg\right\},\cr
& \mathscr{G}=(0\;\;\delta_0):D(\mathscr{A}_m)\to \mathscr{U},\cr
&\mathscr{M}(\begin{smallmatrix}x\\g\end{smallmatrix})=\left(\begin{smallmatrix}Mx\\0\end{smallmatrix}\right),\quad (\begin{smallmatrix}x\\g\end{smallmatrix})\in\mathscr{H}, \cr & \mathscr{N}:=(\mathscr{C}\;\;0):D(\mathscr{A}_m)\to \mathscr{Y},
\end{align*}
where the operators $Q_m$ and $L$ are defined in Example \ref{shift-example}. Using these operators and  the system \eqref{Shift-BC}, the initial input delay system  \eqref{stoc-inputoutput} becomes
\begin{align}\label{Trans-1}
\begin{cases}dZ(t)=\mathscr{A}_m Z(t)dt+\mathscr{M}(Z(t))dW(t),\quad Z(0)=(\begin{smallmatrix}\xi\\\varphi\end{smallmatrix}),& t>0,\cr \mathscr{G}Z(t)=U(t),& t\ge 0,\cr Y(t)=\mathscr{N}Z(t),& t\ge 0.\end{cases}
\end{align}
In order to master the system \eqref{Trans-1}, it is more convenient to reformulate it as a distributed system on $\mathscr{H},\mathscr{U}$ and $\mathscr{Y}$. We will use the work \cite{Greiner}.  We need the following lemmas.
\begin{lemma}\label{L1}
 The operator $\mathscr{G}:D(\mathscr{A}_m)\to \mathscr{U}$ is surjective.
\end{lemma}
\begin{proof}Let $v\in \mathscr{U}$ and set $g=e_\la v$, so $g\in W^{1,2}([-r,0],\mathscr{U})$ and $\delta_0 g=v$. On the other hand, as $G$ is surjective and $Lg\in U$ then there exists $x\in\mathscr{Z}$ such that $Gx=Lg$. This proves that $(\begin{smallmatrix}x\\g\end{smallmatrix})\in D(\mathscr{A}_m)$ and $\mathscr{G}(\begin{smallmatrix}x\\g\end{smallmatrix})=v$.
\end{proof}
\begin{lemma}\label{L2}The operator $\mathscr{A}_m$ is closed.
\end{lemma}\begin{proof}Let $(\begin{smallmatrix}x_n\\g_n\end{smallmatrix})\in D(\mathscr{A}_m)$ and $(\begin{smallmatrix}x\\g\end{smallmatrix}),(\begin{smallmatrix}\overline{x}\\\overline{g}\end{smallmatrix})\in \mathscr{H}$ such that $\|x_n-x\|+\|g_n-g\|_2\to 0$  and $\|A_m x_n-\overline{x}\|+\|g'_n-\overline{g}\|_2\to 0$ as $n\to\infty$. As $(A_m,\mathscr{Z})$ and $(\frac{d}{d\theta},W^{1,2}([-r,0],\mathscr{U}))$ are closed then $x\in \mathscr{Z},$ $g\in W^{1,2}([-r,0],\mathscr{U}),$ and $A_mx=\overline{x},$ $g'=\overline{g}$. Moreover, for any $n$ $Gx_n=Lg_n$ and we can write $x_n=x_n^0+x_n^1$ with $x_n^0\in D(A)$ and $x_n^1\in \ker(\la-A_m)$. Thus $Gx_n^1=Lg_n,$ so that $x_n^1=\D_\la Lg_n$ for any $n$. Let show that $Lg_n\to Lg$ as $n\to\infty$. In fact, we know that there exists a set $N\subset [-r,0]$ with Lebsegue measure $\la(N)=0$ and $g_{n_k}(\theta)\to g(\theta)$ for any $\theta\in [-r,0]\backslash N$ as $k\to\infty$. Now choose $\theta_0\in [-r,0]\backslash N$ and write
\begin{align}\label{waw}
g_{n_k}= e_0 g_{n_k}(\theta_0)+h_{n_k}\quad \text{and}\quad g=e_0 g(\theta_0)+h
\end{align}
where
\begin{align*}
h_n(\theta)=\int^{\theta_0}_\theta (-g_n')(\si)d\si\quad\text{and}\quad h(\theta)=\int^{\theta_0}_\theta (-g')(\si)d\si.
\end{align*}
According to \eqref{waw},
\begin{align*}
\|Lg_{n_k}-Lg\|\le |\mu|([-r,0])\left (\|g_{n_k}(\theta_0)- g(\theta_0)\|+\|h_{n_k}-h\|_\infty\right).
\end{align*}
Using H\"older inequality, we obtain
\begin{align*}
\|h_n-h\|_\infty \le r^{\frac{1}{2}} \|g_n'-g'\|_{L^2([-r,0],\mathscr{U})}\to 0 \quad (n\to\infty).
\end{align*}
This implies that $Lg_{n_k}\to Lg$ as $k\to\infty$, so that $\D_\la Lg_n\to \D_\la Lg$. Thus $x_{n_k}^1\to \D_\la Lg$, so that $x_{n_k}^0\to x-\D_\la Lg$ as $k\to\infty$. On the other hand, $Ax_{n_k}^0=A_m x_{n_k}-A_m x_{n_k}^1=A_m x_{n_k}-\la x_{n_k}^1$ for any $k$. This implies that $Ax_{n_k}^0\to A_m x-\la \D_\la Lg$. As $A$ is closed then in particular $x-\D_\la Lg\in D(A)=\ker G,$ with means that $Gx=Lg$. Then $(\begin{smallmatrix}x\\g\end{smallmatrix})\in D(\mathscr{A}_m)$ and $\mathscr{A}_m(\begin{smallmatrix}x\\g\end{smallmatrix})=(\begin{smallmatrix}\overline{x}\\\overline{g}\end{smallmatrix})$. This shows that $\mathscr{A}_m$ is closed.
\end{proof}

Let $\L$ as in Example \ref{shift-example}.
\begin{proposition} \label{mathscrA}
Assume that $B$ is an admissible control operator for $A$. The following operator
\begin{align}\label{opera-mathscrA}
\mathscr{A}:=\left(\mathscr{A}_m\right)_{|D(\mathscr{A})}\quad \text{with}\quad D(\mathscr{A}):=\ker\mathscr{G}
\end{align}
generates a strongly continuous semigroup $\mathscr{T}:=(\mathscr{T}(t))_{t\ge 0}$ on $\mathscr{H}$ given by
\begin{align}\label{mathscrT}
\mathscr{T}(t)=\begin{pmatrix} T(t)& R(t)\\ 0& S(t)\end{pmatrix},\quad t\ge 0,
\end{align}
where $R(t):L^2([-r,0],\mathscr{U})\to H$ is defined by
\begin{align*}
R(t)g=\int^t_0 T_{-1}(t-s)B \L_{\Lambda}S(s)gds.
\end{align*}
\end{proposition}
\begin{proof}
Let us prove that the family of operators $(\mathscr{T}(t))_{t\ge 0}$ define a strongly continuous semigroup on $\mathscr{H}$. Clearly $\mathscr{T}(0)=I_{\mathscr{H}}$. On one hand, the strong continuity of this family follows from the fact that $R(t)f\to 0$ as $t\to 0^+$, for any $f\in L^2([-r,0],\mathscr{U})$, and the fact that $T$ and $S$ are $C_0$--semigroups on $H$ and $L^2([-r,0],\mathscr{U})$, respectively. On the other hand, for any $t,s\ge 0,$ we have
\begin{align*}
\mathscr{T}(t)\mathscr{T}(s)=\begin{pmatrix}T(t+s)& T(t)R(s)+R(t)S(s)\\ 0& S(t+s)\end{pmatrix}.
\end{align*}
Moreover, a simple change of variables shows that $T(t)R(s)+R(t)S(s)=R(t+s)$. Thus $\mathscr{T}(t)\mathscr{T}(s)=\mathscr{T}(t+s)$. Let $E:D(E)\subset \mathscr{H}\to \mathscr{H}$ be the generator of $\mathscr{T}$ and prove that $E=\mathscr{A}$. Let $\la>0$ be sufficiently large. By taking Laplace Transform in both sides of \eqref{mathscrT}, we obtain
\begin{align}\label{resolventE}
R(\la,E)=\begin{pmatrix}R(\la,A)&\D_\la \L R(\la,Q)\\ 0& R(\la,Q)\end{pmatrix}.
\end{align}
Let now $(\begin{smallmatrix}x\\ f\end{smallmatrix})\in D(\mathscr{A})$ and $(\begin{smallmatrix}x_0\\ f_0\end{smallmatrix})\in \mathscr{H}$ such that
\begin{align*}
(\la-\mathscr{A})(\begin{smallmatrix}x\\ f\end{smallmatrix})=(\begin{smallmatrix}x_0\\ f_0\end{smallmatrix}).
\end{align*}
In particular, we have $(\la-A_m)x=x_0$ and $f=R(\la,Q)f_0$. As $x\in \mathscr{Z},$ then by \eqref{Am=A+BG}, we have $(\la-A_{-1})x-BGx=x_0$. But $Gx=Lf=LR(\la,Q)f_0$ and $\D_\la=R(\la,A_{-1})B$. Then $x=R(\la,A)x_0+\D_\la LR(\la,Q)f_0$. This implies that $\la-\mathscr{A}$ is invertible and $(\la-\mathscr{A})^{-1}=R(\la,E),$ due to \eqref{resolventE}. Hence $\mathscr{A}=E$. This ends the proof.
\end{proof}
The following result is needed in the proof of the main result of this work.
\begin{proposition}\label{P-admis}
  Assume that the triple $(A,B,C)$ is regular. The following operator
  \begin{align*}
  \mathscr{P}:=\mathscr{N}_{|D(\mathscr{A})}\in\calL(D(\mathscr{A}),\mathscr{U})
  \end{align*}
  is an admissible observation operator for $\mathscr{A}$. Furthermore, if $\mathscr{P}_\Lambda,\,C_\Lambda$ and $\L_\Lambda$ denote Yosida extension of $\mathscr{P},C$ and $\L$, for $\mathscr{A},$ $A$ and $Q,$ respectively, then
  \begin{align}\label{walid}
  D(C_\Lambda)\times D(\L_\Lambda)\subset D(\mathscr{P}_\Lambda)\quad\text{and}\quad \left(\mathscr{P}\right)_{|D(C_\Lambda)\times D(\L_\Lambda)}=(C_\Lambda\;\;0).
  \end{align}
\end{proposition}
\begin{proof}
First, we give some information about the Yosida extension of $\mathscr{P}$ for $\mathscr{A}$. Let $\la>0$ be sufficiently large and let $(\begin{smallmatrix} x\\ g\end{smallmatrix})\in D(C_\Lambda)\times D(\L_\Lambda)$. According to \eqref{resolventE},  we have
\begin{align*}
\mathscr{P}\la R(\la,\mathscr{A})(\begin{smallmatrix}x\\g\end{smallmatrix})=C\la R(\la,A)+C\D_\la \L\la R(\la,Q)g,
\end{align*}
 \begin{align*}
\|C\D_\la \L\la R(\la,Q)g\| \le \|C\D_\la\| \left(\|\L\la R(\la,Q)g-\L_\Lambda g\|+\|\L_\Lambda g\|\right).
\end{align*}
As the transfer function of the regular system $(A,B,C)$ satisfies $C\D_\la\to 0$ as $\la\to+\infty$, the property \eqref{walid} follows. Second, we show that $\mathscr{P}$ is an admissible observation operator for $\mathscr{A}$. In fact, let $(\begin{smallmatrix} x\\ g\end{smallmatrix})\in D(\mathscr{A})$ and $\al>0$. Then
\begin{align*}
\int^\al_0 \|\mathscr{P}\mathscr{T}(t)(\begin{smallmatrix} x\\ g\end{smallmatrix})\|^2dt=\int^\al_0 \|\mathscr{P}_\Lambda\mathscr{T}(t)(\begin{smallmatrix} x\\ g\end{smallmatrix})\|^2dt
\end{align*}
As $(A,B,C)$ is regular and $\L$ is an admissible observation operator for $Q,$ then
\begin{align*}
& T(t)x\in D(C_\Lambda),\quad R(t)g= \Phi^{A,B}_t \L_\Lambda S(\cdot)g \in D(C_\Lambda),\quad S(t)g\in D(\L_\Lambda),
\end{align*}
for a.e. $t\ge 0$. Thus $\mathscr{T}(t)(\begin{smallmatrix} x\\ g\end{smallmatrix})\in D(C_\Lambda)\times D(\L_\Lambda)$ for a.e. $t\ge 0$, due to \eqref{mathscrT}. Now according to \eqref{walid}, we have
\begin{align*}
\int^\al_0 \|\mathscr{P}\mathscr{T}(t)(\begin{smallmatrix} x\\ g\end{smallmatrix})\|^2dt&=\int^\al_0 \|C_\Lambda T(t) x+ C_\Lambda \Phi_t^{A,B} \L_\Lambda S(\cdot)g\|^2dt\cr & \le 2 \ga \|x\|^2+2\kappa \|\L_\Lambda S(\cdot)g\|^2_{L^2([0,\al],\mathscr{U})}\cr & \le 2(\ga+\kappa c)\left(\|x\|^2+\|g\|^2_{L^2([-r,0],\mathscr{U})}\right)\cr & \le \tilde{c} \left\|(\begin{smallmatrix} x\\ g\end{smallmatrix})\right\|^2,
\end{align*}
where $\ga,\kappa,c$ and $\tilde{c}$ are constant, some of them depend on $\al$.
 \end{proof}

The following is the main result of this section.

\begin{theorem}\label{Main-1}
Assume that $B$ is an admissible control operator for $A$. Then there exists an admissible control operator $\mathscr{B}\in\calL(\mathscr{U},\mathscr{H}_{-1})$ for $\mathscr{A}$ such that the system \eqref{Trans-1} can be rewritten as
\begin{align}\label{Trans-2}
\begin{cases} dZ(t)=\left(\mathscr{A} Z(t)+\mathscr{B}U(t)\right)dt+\mathscr{M}(Z(t))dW(t),\quad Z(0)=(\begin{smallmatrix}\xi\\\varphi\end{smallmatrix}),& t>0,\cr Y(t)=\mathscr{P}Z(t),& t>0.
\end{cases}
\end{align}
  Moreover this system has a unique mild solution $Z\in \calC_\F(0,+\infty;L^2(\Om,\mathscr{H}))$ satisfying
  \begin{align}\label{mild-Z}
  Z^{\xi,\varphi,U}(t)=\mathscr{T}(t)(\begin{smallmatrix}\xi\\\varphi\end{smallmatrix})+\int^t_0\mathscr{T}(t-s)\mathscr{M}(Z^{\xi,\varphi,0}(s))dW(s)+\Phi^{W}_t U
  \end{align}
  for any $U\in L^2_{\F}(0,+\infty;\mathscr{U})$, where $\Phi^{W}_t\in \calL(L^2_{\F}(0,+\infty;\mathscr{U}),L^2_{\mathscr{F}_t}(\Om,\mathscr{H}))$ for any $t\ge 0$ and
  \begin{align}\label{Phi-W}
  \Phi^{W}_t U=\Phi^{\mathscr{A},\mathscr{B}}_t U+\int^t_0 \mathscr{T}(t-s)\mathscr{M}(\Phi^{W}_s U)dW(s),
  \end{align}
  where $(\Phi^{\mathscr{A},\mathscr{B}}_t)_{t\ge 0}$ is the family of control maps associated with the pair $(\mathscr{A},\mathscr{B})$. Furthermore, if the triple $(A,B,C)$ is regular and if $X^{\xi,\varphi,U}(t)$ is the first component of $Z^{\xi,\varphi,U}(t)$, then
$X^{\xi,\varphi,U}(t)\in D(C_\Lambda)$ for a.e. $t>0$ and $\P$-a.s. and
\begin{align*}
\|C_\Lambda X(\cdot)\|_{L^2_{\F}(0,\al;\mathscr{Y})}\le c(\al)\left( \|\xi\|_{L^2_{\mathscr{F}_0}(\Om,\mathscr{H})}+\|U\|_{L^2_{\F}(0,\al;\mathscr{U})}\right)
\end{align*}
for any $\al>0,$ $\xi\in L^2_{\mathscr{F}_0}(\Om,H),$ $\varphi\in L^2_{\mathscr{F}_0}(\Om,L^2([-r,0],\mathscr{U})),$ $U\in L^2_{\F}(0,\al;\mathscr{U})$ and some constant $c(\al)>0$.
\end{theorem}

\begin{proof}
Using the technique of Section \ref{sec:2}, we will reformulate the stochastic boundary problem \eqref{Trans-1} to a distributed linear input-output stochastic system. In fact, let $\mathscr{D}_\la,\;\la\in\rho(\mathscr{A})=\rho(A),$ be the Dirichlet operator associated with $\mathscr{A}_m$ and $\mathscr{G}$, and select
\begin{align}\label{mathscrB}
  \mathscr{B}:=(\la-\mathscr{A}_{-1})\mathscr{D}_\la,\quad \la\in\rho(\mathscr{A}).
\end{align}
This implies $\mathscr{A}_m=\mathscr{A}_{-1}+\mathscr{B}\mathscr{G}$ on $D(\mathscr{A}_m)$ (see Section \ref{sec:2}). Thus the system \eqref{Trans-1} is transformed into the system \eqref{Trans-2}. Let us prove that $\mathscr{B}$ is an admissible control operator for $\mathscr{A}$. To this end, we first compute the expression of $\mathscr{D}_\la$ for $\la\in\rho(\mathscr{A})$. Let $(\begin{smallmatrix}x\\g\end{smallmatrix})\in \ker(\la-\mathscr{A}_m)$.
This means that $x\in\mathscr{Z}$ and $g\in W^{1,2}([-r,0],\mathscr{U})$ are such that $Gx=Lg$ $(\la-\mathscr{A}_m)x=0$ and $\la g-g'=0$. This proves that $x=\mathscr{D}_\la v$ with $v\in\mathscr{U}$ and $g=e_\la g(0)$. We also have, $v=G\mathscr{D}_\la v=Le_\la g(0)$. We then deduce that
\begin{align}\label{new-dirichlet}
\mathscr{D}_\la v=\begin{pmatrix} \D_\la Le_\la v\\ e_\la v\end{pmatrix},\quad \la\in\rho(A).
\end{align}
Let $(Q,\beta,\L)$ be the regular system studied in Example \ref{shift-example} with control maps $(\Phi^{Q,\beta}_t)_{t\ge 0}$ and extended input operator ${\bf F}^{Q,\beta,\L}:={\bf F}^{L}$. Using \eqref{BC}, \eqref{mathscrB}, \eqref{new-dirichlet} and the injectivity of the Laplace transform, we have
\begin{align}\label{mathscrAB}
\Phi^{\mathscr{A},\mathscr{B}}_t u:=\int^t_0 \mathscr{T}_{-1}(t-s)\mathscr{B}u(s)ds=\begin{pmatrix} \Phi^{A,B}_t{\bf F}^Lu\\ \Phi^{Q,\beta}_tu\end{pmatrix}
\end{align}
for any $t\ge 0$ and $u\in L^2([0,+\infty),\mathscr{U})$. The admissibility of $\mathscr{B}$ for $\mathscr{A}$ follows immediately from those of $B$ and $\beta$ for $A$ and $Q,$ respectively. Now according to \cite[Theorem 2.7]{Lahbiri-Hadd} (see also \cite{Lu2015}), the system \eqref{Trans-2} has a unique mild solution $Z\in \calC_{\F}(0,+\infty,L^2(\Om,\mathscr{H}))$ satisfying \eqref{Phi-W}. According to Proposition \ref{P-admis}, the operator $\mathscr{P}$ is an admissible observation operator for $\mathscr{A}$. Next we show that the triple $(\mathscr{A},\mathscr{B},\mathscr{P})$ is regular. In fact,  the regularity of the triple $(A,B,C)$ is regular implies that $\Phi^{A,B}_t u\in D(C_\Lambda)$ for a.e. $t\ge 0$ and $u\in L^2_{loc}(\R^+,\mathscr{U})$. Similarly for the system associated with the left shift semigroup, $\Phi^{Q,\beta}u\in D(\L_{\Lambda})$  for a.e. $t\ge 0$ and $u\in L^2_{loc}(\R^+,\mathscr{U})$. Now according to \eqref{walid} and \eqref{mathscrAB}, we have $\Phi^{\mathscr{A},\mathscr{B}}_t u\in D(\mathscr{P}_\Lambda)$ for a.e. $t\ge 0$ and $u\in L^2_{loc}(\R^+,\mathscr{U})$ and the operator
\begin{align*}
\left({\bf F}^{\mathscr{A},\mathscr{B},\mathscr{P}} u\right)(t)=\mathscr{P}_\Lambda \Phi^{\mathscr{A},\mathscr{B}}_t u= \left({\bf F}^{A,B,C}{\bf F}^Lu\right)(t),
\end{align*}
is linear bounded from $L^2([0,\al],\mathscr{U})$ to $L^2([0,\al],\mathscr{U})$ for any $\al>0$. This shows that the triple $(\mathscr{A},\mathscr{B},\mathscr{P})$ is well-posed. Again by regularity of triples $(A,B,C)$ and $(Q,\beta,\L)$, we have  ${\rm Range}(\D_\la)\subset D(C_\La)$ and ${\rm Range}(e_\la)\subset D(\L_\La)$ for some (hence all) $\la\in \rho(A)$. This shows that ${\rm Range}(\mathscr{D}_\la)\subset D(\mathscr{P}_\La)$ for $\la\in\rho(\mathscr{A})$, due to \eqref{walid}. Thus $(\mathscr{A},\mathscr{B},\mathscr{P})$ is regular. Now, by \cite[Theorem 2.15]{Lahbiri-Hadd}, we have $Z(t):=Z^{\xi,\varphi,U}(t)\in D(\mathscr{P}_\Lambda)$ for a.e. $t>0$ and $\P$-a.s., and
\begin{align}\label{Z-estim-1}
\|\mathscr{P}_{\Lambda}Z^{\xi,\varphi,U}(\cdot)\|_{L^2_{\F}(0,\al;\mathscr{U})}\le c(\al)\left( \|(\begin{smallmatrix}\xi\\\varphi\end{smallmatrix})\|_{L^2_{\mathscr{F}_0}(\Om,\mathscr{H})}+\|U\|_{L^2_{\F}(0,\al;\mathscr{U})}\right)
\end{align}
for any $\al>0,$ $\xi\in L^2_{\mathscr{F}_0}(\Om,H),$ $\varphi\in L^2_{\mathscr{F}_0}(\Om,L^2([-r,0],\mathscr{U})),$ $U\in L^2_{\F}(0,\al;\mathscr{U})$ and some constant $c(\al)>0$.  We denote by
\begin{align*}
Z^{\xi,\varphi,U}(t)=\begin{pmatrix}X^{\xi,\varphi,U}(t)\\ V(t,\cdot)\end{pmatrix},\qquad t\ge 0.
\end{align*}
Then, by using Proposition \ref{mathscrA}, the equations \eqref{mild-Z}, \eqref{Phi-W} and the relation \eqref{mathscrAB}, we have
\begin{align*}
&X^{\xi,\varphi,U}(t)=T(t)\xi+\int^t_0 T_{-1}(t-s)B\L_\Lambda U_sds+\int^t_0 T(t-s)M(X^{\xi,\varphi,U}(s))dW(s),\cr
& V(t,\cdot)=U_t
\end{align*}
for a.e. $t\ge 0$ and $\P$-a.s. Again by \cite{Lahbiri-Hadd}, we have $X^{\xi,\varphi,U}(t)\in D(C_\Lambda)$ and $V(t,\cdot)\in D(\L_\Lambda)$ for a.e. $t>0$ and $\P$-a.s. This implies that $Z^{\xi,\varphi,U}(t)\in D(C_\Lambda)\times D(\L_\Lambda)$ for a.e. $t>0$ and $\P$-a.s.  Thus the results follow by \eqref{walid} and \eqref{Z-estim-1}.
\end{proof}
%\begin{example}
% Let $\mathscr{O}\subset \mathbb{R}^n$ be an open bounded set and put $X=L^2(\Omega)$ with boundary $\partial\mathscr{O}$. We consider the following nonlinear initial value problem
%\begin{eqnarray}\label{exx}
%	\begin{cases}
%		\dot{z}(t,x)=\Delta z(t,x),\quad z(0,x)=f(x)  & x\in \mathscr{O},t>0,\\
%		 y(t,x)=\displaystyle\int_{\mathscr{O}}K(x,y)(-\Delta)^{\beta}z(t,y)dy, & x\in \partial \mathscr{O}, t\ge 0,
%	\end{cases}
%\end{eqnarray}
%for $f\in L^2(\mathscr{O})$, $\beta\in (0,\frac{1}{2})$ and  $K\in L^\infty( \partial \mathscr{O}\times \mathscr{O})$.
%\end{example}
\begin{example}\label{heat-sans-delay}
Consider the following input-output partial differential stochastic system
\begin{align}\label{heat-eq}
\begin{cases}
\partial_t X(t,x)=\Delta X(t,x)+\displaystyle\int_{\mathscr{O}}k(x,y)X(t,y)dy \partial_t W(t),& x\in\mathscr{O},\quad t>0,\cr X(0,x)=\xi(x),& x\in\mathscr{O},\cr  \nabla X(t,x)|\nu(x)=U(t-r,x),& x\in \partial\mathscr{O},\quad t>0,\cr U(t,x)=\varphi(t,x),& t\in [-r,0],\; x\in \mathscr{O},\cr Y(t,x)=c(x)X(t,x),& x\in \partial\mathscr{O},\quad t\ge 0.
\end{cases}
\end{align}
Here, $\mathscr{O}$ is a bounded open subset of $\R^n$ with a $C^2$-boundary $\partial\mathscr{O}$ and outer unit normal $\nu(x)$, the initial condition $\xi\in L^2(\Om\times \mathscr{O}),$ $r>0$ is a real number, $k\in L^2(\mathscr{O}\times \mathscr{O}),$ and $c(\cdot)\in C_b(\partial\mathscr{O})$. We set
\begin{align*}
H=L^2(\mathscr{O}),\quad \mathscr{Z}=H^2(\mathscr{O}),\quad \mathscr{U}=\mathscr{Y}=L^2(\partial\mathscr{O}).
\end{align*}
For any $f\in H,$ we define
\begin{align*}
\left[M(f)\right](x)=\int_{\mathscr{O}}k(x,y)f(y)dy,\qquad x\in  \mathscr{O}.
\end{align*}
Using H\"older's inequality, we get $\|M(f)\|_H \le \|k\|_{L^2(\mathscr{O}\times \mathscr{O})} \|f\|_H$. We also define the delay operator
\begin{align*}
Lg=g(-r),\qquad g\in W^{1,2}([-r,0],\mathscr{U}).
\end{align*}
Let us now select the following operators
\begin{align*}
& A_m:=\Delta, \quad D(A_m):=\mathscr{Z}, \cr  \quad & G: \mathscr{Z}\to \mathscr{U},\quad  (G \psi)(x)=\nabla \psi(x)|\nu(x),\cr & \mathscr{C}:\mathscr{Z}\to \mathscr{Y},\quad (\mathscr{C}\psi)(x)=c(x)\psi(x).
\end{align*}
We know that the operator
\begin{align*}
A:=A_m,\qquad  D(A):=\{f\in \mathscr{Z}: (G \psi)(x)=\nabla \psi(x)|\nu(x)=0,\;x\in\partial\mathscr{O}\}
\end{align*}
generates an analytic semigroup $T:=(T(t))_{t\ge 0}$ on $H$. In what follows, we collect from \cite[Chap.3]{LT-00} some facts about fractional operators and their relation to Neumann boundary conditions.  We denote by $\phi=\D_1\psi\in W^{3/2,2}(\mathscr{O})$ the solution of the elliptic boundary-value problem $\Delta \phi=0$ on $\mathscr{O}$ and $\nabla\phi|\nu=\psi$ on $\partial\mathscr{O}$, where $\psi\in L^2(\partial\mathscr{O})$. Then $\D_1$ maps $L^2(\partial\mathscr{O})$ continuously into $D((I-A)^\beta)$ for $\beta\in (0,\frac{3}{4})$. On the other hand, by using the analyticity of the semigroup $T,$ we obtain
\begin{align}\label{yaya}
\|(I-A)^\al T(t-s)(I-A_{-1})\D_1 v\|_H \le c_\beta (t-s)^{\beta-\al-1}\|v\|_{L^2(\partial\mathscr{O})}
\end{align}
for any $v\in L^2(\partial\mathscr{O}),\;\al\ge 0$ and $\beta\in (0,\frac{3}{4}),$ where $c_\beta>0$ is a constant. In particular for $\al=0$ and $\beta\in (\frac{1}{2},\frac{3}{4})$ the control maps $\Phi^{A,B}_t$ associated to the control operator $B:=(I-A_{-1})\D_1$ satisfy  $\Phi^{A,B}_t: L^2([0,+\infty),\mathscr{U})\to H$  and $\|\Phi^{A,B}_t\|\le \delta_\beta t^{2\beta-1}$ for any $t\ge 0$ and a constant $\delta_\beta>0$. This shows that $(A,B)$ is well-posed. On the other hand, the observation operator $\mathscr{C} \varphi= c(\cdot){\rm tr}\varphi$, where ${\rm tr}$ is the trace operator, is uniformly bounded from $D((I-A)^\al)$ to $L^2(\partial\mathscr{O})$ for $\al>\frac{1}{4}$. Now if we set $C:=\mathscr{C}$ with domain $D(C)=D(A)$, then for any $\varphi\in D(A)$, $\|CT(t)\varphi\|\le \ga t^{-\al}$ for a constant $\ga>0$. Thus by choosing $\al\in (\frac{1}{4},\frac{1}{2})$, we deduce that $(C,A)$ is well-posed. To apply theorem \ref{Main-1}, we need to prove that the triple $(A,B,C)$ is regular. In fact, for  $\al\in (\frac{1}{4},\frac{1}{2}),$ $\beta\in \frac{1}{2},\frac{3}{4})$ and $v\in W^{2,2}([0,t],\mathscr{U})$ with $v(0)=0$, we have
\begin{align}\label{F-estim}
\|{\bf F}^{A,B,C}v=\mathscr{C}\Phi^{A,B}_{\cdot}v\|_{L^2([0,t],\mathscr{U})} \le \kappa_\beta t^{\beta-\al} \|v\|_{L^2([0,t],\mathscr{U})}
\end{align}
for a constant $\kappa_\beta>0,$ due to \eqref{yaya}. Now by Proposition \ref{abc-well-posed}, the triple $(A,B,C)$ is well-posed. On the other hand, let $v_0\in \mathscr{U}$. By H\"older's inequality and the estimate \eqref{F-estim}, we obtain
\begin{align*}
\left\| \frac{1}{t} \int^t_0 \left({\bf F}^{A,B,C}\1_{[0,+\infty)}v_0\right)(s)ds\right\|_{\mathscr{Y}}\le \kappa_\beta t^{\beta-\al}\|v_0\|_{\mathscr{U}}\underset{t\to 0^+}{\longrightarrow}0.
\end{align*}
Thus $(A,B,C)$ is regular. By a simple computation one can verify that the operator $L$ is an admissible observation operator for the left shift semigroup on $L^2([-r,0],H)$. Thus the assumption of Theorem \ref{Main-1} are satisfied for the stochastic input delay system \eqref{heat-eq}, so it is a well-posed system.
\end{example}

\section{Application: a stochastic Schr\"odinger
system}
In this section, we are concerned with the well-posedness of a stochastic Schr\"odinger
equation with a partial Dirichlet input delay control and a collocated observation. We first need some notations and definitions in order to describe the system that will be studied.

Let $\mathscr{O}\subset \R^n$ ($n\ge 2$) be an open bounded region with $C^3$-boundary
$\partial\mathscr{O}=\overline{\Gamma}_0\cup \overline{\Gamma}_1$,
 where $\Gamma_0$ and $\Gamma_1$ are disjoint
parts of the boundary relatively open in $\partial\mathscr{O}$ and ${\rm int}(\Gamma_0)\neq \emptyset$. Let $H=H^{-1}(\mathscr{O})$ (the dual space
of the Sobolev space $H^1_0(\mathscr{O})$ with respect to the pivot
space $L^2(\mathscr{O})$). We define a bilinear form $a(\cdot,\cdot)$ on
$H^1_0(\mathscr{O})$ by
\begin{align*}
a(f,g):=\int_{\mathscr{O}}\nabla f(x) \overline{\nabla g(x)}dx,\qquad f,g\in H^1_0(\mathscr{O}).
\end{align*}
To this form, we associate the following linear operator $\A:D(\A):=H^1_0(\mathscr{O})\to H,$
\begin{align*}
\langle \A f,g\rangle=a(f,g),\qquad f,g\in H^1_0(\mathscr{O}),
\end{align*}
see e.g. \cite{Arendt-et-al}, \cite{Ouhabaz}. Clearly $\A$ defines a canonical isomorphism from $D(\A)$ to $H$. Now if we denote by $-\Delta: H^2(\mathscr{O})\cap  H^1_0(\mathscr{O})\to L^2(\mathscr{O})$ the Laplacian operator, then $\A$ coincides with $-\Delta$ on $H^2(\mathscr{O})\cap  H^1_0(\mathscr{O})$. This implies that $\A^{-1}f=(-\Delta)^{-1}f$ for any $f\in L^2(\mathscr{O})$. We also recall (see e.g. \cite{Komornik}) that $D(\A^{\frac{1}{2}})=L^2(\mathscr{O})$ and $\A^{\frac{1}{2}}$  is a canonical isomorphism from $L^2(\mathscr{O})$ to $H$. We also have the following continuous embedding:
\begin{align*}
D(\A)\subset D(\A^{\frac{1}{2}})=L^2(\mathscr{O})\hookrightarrow H=H'\hookrightarrow (D(\A^{\frac{1}{2}}))'\subset D(\A)'.
\end{align*}
Now consider the system
\begin{align}\label{chro}
\begin{cases}
dX(t,x)=-i\Delta X(t,x)dt+q(x)X(t,x)dW(t),
& t> 0,\;x\in\mathscr{O},\cr X(0,x)=\xi(x),& x\in \mathscr{O},\cr X(t,x)=0,
& t\ge 0,\; x\in \Gamma_1,\cr X(t,x)=\displaystyle\int^0_{-r}d\mu(\theta,x)U(t+\theta,x), & t\ge 0,\; x\in\Gamma_0,\cr
U(t,x)=\varphi(t,x),& -r\le t\le 0,\; x\in\Gamma_0,\cr
Y(t,x)=i\displaystyle\frac{\partial (\Delta^{-1}X)}{\partial \nu},& t\ge 0,\; x\in \Gamma_0,
\end{cases}
\end{align}
where
$\nu$ is the unit normal of $\partial\mathscr{O}$ pointing towards the exterior of $\mathscr{O}$, $q\in L^\infty(\mathscr{O})$, and for  $\{\mu(\theta,x): \theta\in [-r,0],\;x\in \Gamma_0\}\subset\calL(\R^n)$ and $\mu(\theta,\cdot)\in\calL(L^2(\Gamma_0))$  is a function of bounded variation on $\theta$.
 %for all $f\in L^2(\Gamma_0)$,
%\begin{align*}
%\int_{\Gamma_0}\|\mu(\theta,x)f(x)\|^2dx\le c \|f\|^2_{L^2(\Gamma_0)}
%\end{align*}
%for a constant $c>0$.
Denote $\mathscr{U}:=L^2(\Gamma_0)$. We consider the following closed operator
\begin{align*}
A_m:=i(-\Delta),\quad D(A_m)=\{f\in H^2(\mathscr{O}): f_{|\Gamma_1}=0\}.
\end{align*}
We define the following operators
\begin{align*}
&\mathscr{M}g=q g,\qquad g\in H^{-1}(\mathscr{O}),\cr
& G f:=f_{|\Gamma_0},\qquad f\in D(A_m),\cr & \mathscr{C} f=i\displaystyle\frac{\partial (\Delta^{-1}f)}{\partial \nu},\qquad f\in D(A_m).
\end{align*}
With these notations, the system \eqref{chro} takes the form of our model system \eqref{stoc-inputoutput}. Thus to show the well-posednees of the system \eqref{chro} it suffices to prove that the assumptions of Theorem \ref{Main-1} hold. In fact, we define $A:=(A_m)_{|D(A)}$ with $D(A)=\ker G$. Thus we have
\begin{align*}
A=i\A, \qquad D(A)=\{f\in H^2(\mathscr{O}): f_{|\partial \mathscr{O}}=0\}.
\end{align*}
Let $\D_0$ (we take $\la=0$) be the Dirichlet operator associated with $A_m$ and $G$. It is determined as follow: $\D_0 u=v$ if and only if $\Delta v=0$ in $\mathscr{O},$ $v=0$ on $\Gamma_1$ and $v=u$ on $\Gamma_0$. We have $\D_0\in\calL(\mathscr{U},L^2(\mathscr{O}))$. We select $B=(-A_{-1})\D_0=-i\A_{-1}\D_0$, where $\A_{-1}\in\calL\left(D(\A^{\frac{1}{2}}),(D(\A^{\frac{1}{2}})')\right)$ is the extension of $\A$. Let us show that $B\in\calL(\mathscr{U},(D(\A^{\frac{1}{2}})')$ is an admissible control operator for $A$. To this end, it suffice to prove that its adjoint $B^\ast$ is an admissible observation operator for $A$. As shown in \cite{Guo-sha}, we have
\begin{align*}
B^\ast f=\left( i\displaystyle\frac{\partial (\Delta^{-1}f)}{\partial \nu}\right)_{|\Gamma_0}.
\end{align*}
Otherwise, $B^\ast=\mathscr{C}_{|D(A)}$. According to \cite[Theorem 1.2]{Guo-sha}, the triple $(A,B,B^\ast)$ is regular. Thus, by Theorem \ref{Main-1}, the stochastic
system \eqref{chro} is well-posed.
\begin{remark}
 We can choose the potential $q$ in \eqref{chro} to be a stochastic process. In this case, the result of this section remains the same if we assume that $q\in L^\infty(\Om\times \mathscr{O})$. In fact, we define a multiplication operator on $L^2(\Om,H)$ by $(\tilde{\mathscr{M}}f)(\om,x)=q(\om,x)f(\om)$ for $x\in\mathscr{O}$ and $\P$-a.e. $\om\in\Om$. As our main results is based on \cite[Theorems 2.7 and 2.15]{Lahbiri-Hadd} and the proofs of these theorems are not affected by this choice, the result of this section will also not be affected.
\end{remark}

\end{document}